\documentclass[dvipdfmx, a4paper,11pt]{amsart}
\usepackage{amsmath,amsthm,amssymb,amsthm, amsfonts, amscd, mathrsfs,mathrsfs, xcolor}
\makeatletter
 \def\@makefnmark{%
 \leavevmode
 \raise.9ex\hbox{\check@mathfonts
 \fontsize\sf@size\z@\normalfont%
 \@thefnmark}%
 }
 \makeatother
\newcommand\diam{\operatorname{diam}}

\newcommand{\N}{\mathbb{N}}

\newcommand{\K}{\mathcal{K}}
\newcommand{\R}{\mathbb{R}}

\newcommand{\C}{\mathfrak{C}}

\newcommand{\LL}{\mathcal{L}}
\newcommand{\MM}{\mathcal{M}}

\newcommand{\U}{\mathcal{U}}
\newcommand{\SO}{\mathcal{SO}}

\newcommand{\CH}{\mathrm{H}_{\mathcal{C}}}
\newcommand{\ISO}{\mathscr{I}_{\mathcal{SO}}}
\newcommand{\ISOs}{\mathscr{I}_{\mathcal{SO}^{\ast}}}
\newcommand{\SF}{\mathcal{SO}_{\sharp}}

\theoremstyle{plain}
\newtheorem{thm}{Theorem}[section]
\newtheorem{lemma}[thm]{Lemma}
\newtheorem{proposition}[thm]{Proposition}
\newtheorem{corollary}[thm]{Corollary}
\theoremstyle{definition}

\newtheorem{remark}[thm]{Remark}
\newcounter{cn}
\title[Lattices of slowly oscillating functions]{Lattices of slowly oscillating functions}
\author{Yutaka Iwamoto}
\thanks{This work was supported by JSPS KAKENHI Grant Number JP24K06726 and by the Research Institute for Mathematical Sciences,
an International Joint Usage/Research Center located in Kyoto University.}
\address{Faculty of Fundamental Science, National Institute of Technology (KOSEN), Niihama College,
Niihama, 792-8580, Japan}
\email{y.iwamoto@niihama-nct.ac.jp}
\begin{document}

\begin{abstract}
  We show that lattice isomorphisms between lattices of slowly oscillating functions on chain-connected proper metric spaces induce coarsely equivalent homeomorphisms. This result leads to a Banach-Stone-like theorem for these lattices. Furthermore, we provide a representation theorem that characterizes linear lattice isomorphisms among these lattices.
\end{abstract}

\maketitle

\section{Introduction}

The lattice structure of a topological space often determines its fundamental properties. 
For example, a lattice isomorphism between lattices of continuous functions on compact spaces induces a homeomorphism between those spaces \cite{Kapl}, 
a result classified as a Banach-Stone-like theorem.

Let $\U(X)$ denote the lattice of all uniformly continuous functions on $X$, with $\U^{\ast}(X)$ as the sublattice consisting of all bounded functions.
For a complete metric space $X$, it is well-known that the lattice structure of $\U^{\ast}(X)$ determines the uniform structure of $X$.
F\'elix Cabello S\'anchez and Javier Cabello S\'anchez proved that
the lattice structure of $\U(X)$ also determines the uniform structure of $X$
among complete metric spaces \cite{FJ_Cabello_Sanchez}.
This result is recognized as a Banach-Stone-like theorem for lattices of uniformly continuous functions.
(For variations of Banach-Stone-like theorems concerning lattices of uniformly continuous functions, see \cite{FJ_Cabello_Sanchez}, \cite{Garrido-Jaramillo-2000}, \cite{Husek}, \cite{H-P-1}, and \cite{H-P-2}.)
In particular, they developed methods for constructing homeomorphisms induced by lattice isomorphisms. 
Thereafter, building on many preceding results, Denny H. Leung and Wee-Kee Tang presented a unified and thorough study of nonlinear order isomorphisms between function spaces \cite{LT}.

This paper focuses on lattices of slowly oscillating real-valued continuous functions on proper metric spaces. Slowly oscillating functions play a crucial role in defining Higson coronas \cite{Roe} and Higson compactifications \cite{Keesling}, and they appear frequently in coarse geometry. 
Despite their importance, the lattices of these functions remain relatively unexplored (cf. \cite{Iwa}).

First we apply results from \cite{LT} to lattice isomorphisms between lattices of slowly oscillating functions. 
Let $\SO(X)$ denote the lattice of slowly oscillating functions on $X$ and $\SO^{\ast}(X)$ the sublattice of bounded functions.  
Then a lattice isomorphism $T: \SO(X) \to \SO(Y)$ with $T(0) = 0$ induces a homeomorphism $\tau: X \to Y$ such that $T(f)(y) = T(g)(x)$ if and only if $f(\tau^{-1}(y)) = g(\tau^{-1}(y))$ for all $f, g \in \SO(X)$ and $y \in Y$ (Theorem \ref{THM_3-6}).

Next, we show that for chain-connected proper metric spaces,
these induced homeomorphisms are coarse homeomorphisms.  
A metric space $X$ is said to be \textit{$r$-chain connected} if, for any two points $x$ and $x'$ in $X$, there exist finitely many points $p_0, \dots, p_n$ in $X$ such that $p_0 = x$, $p_n = x'$, and $d_X(p_{i-1}, p_i) \leq r$ for all $i = 1, \dots, n$. A metric space is called \textit{chain-connected} if it is $r$-chain connected for some $r > 0$.  
It is worth noting that if $X$ is an unbounded chain-connected proper metric space, then its asymptotic dimension is positive, i.e., $\textrm{asdim} X \geq 1$ (see \cite{BD} for asymptotic dimension).  
If $X$ and $Y$ are chain-connected proper metric spaces and $T: \SO(X) \to \SO(Y)$ is a lattice isomorphism, we show that $T$ induces a coarse homeomorphism $\tau: X \to Y$, meaning $\tau$ is a homeomorphism that establishes a coarse equivalence between these spaces. 
The same result holds for a lattice isomorphism $T: \SO^{\ast}(X) \to \SO^{\ast}(Y)$.
Furthermore, coarse homeomorphisms induce homeomorphisms between the Higson compactifications.
Therefore, the lattice structures of $\SO(X)$ and $\SO^{\ast}(X)$ determine the topology of the Higson compactification $hX$ of $X$. 
These results lead to a Banach-Stone-like theorem for lattices of slowly oscillating functions (Theorem \ref{THM_4-5}).

Finally, we analyze linear lattice isomorphisms between chain-connected proper metric spaces.
Let $X$ and $Y$ be chain-connected proper metric spaces. 
Define $\ISO(X, Y)$ (resp. $\ISOs(X, Y)$) as the set of all linear lattice isomorphisms from $\SO(X)$ onto $\SO(Y)$ (resp. from $\SO^{\ast}(X)$ onto $\SO^{\ast}(Y)$). 
As in the case of uniformly continuous functions, every linear lattice isomorphism $T: \SO(X) \to \SO(Y)$ is a weighted composition operator, that is, $T(f)(y) = \omega(y) \cdot f(\tau^{-1}(y))$ for every $y \in Y$, where $\omega = T(1) \in \SO(Y)$ and $\tau: X \to Y$ is a homeomorphism induced by $T$. 
We define $\SF(Y)$ (resp. $\SF^{\ast}(Y)$) as the set of all functions $T(1)$ corresponding to linear lattice isomorphisms $T: \SO(X) \to \SO(Y)$ (resp. $T: \SO^{\ast}(X) \to \SO^{\ast}(Y)$). 
Let $\CH(X, Y)$ denote the set of all coarse homeomorphisms from $X$ to $Y$.
Using these definitions,
we obtain a representation theorem for linear lattice isomorphisms (Theorem \ref{THM_4-11}) which asserts that there are bijections $\SF(Y) \times \CH(X, Y) \to \ISO(X,Y)$ and $\SF^{\ast}(Y) \times \CH(X, Y) \to \ISOs(X,Y)$.
In particular, we show that $\ISO(Y,X) \subset \ISOs(Y, X)$,
 and that $\SF(Y)$ is a sublattice of $\SF^{\ast}(Y)$.

\section{Preliminaries}

Let $(X, d_X)$ be a metric space and let $B_{d_X}(x,r)$
denote the closed ball of radius $r$ centered at $x\in X$.
A metric $d_X$ on $X$ is called \textit{proper} if $B_{d_X}(x, r)$ is compact
for every $x\in X$ and $r>0$.
Unless otherwise stated,
$(X, d_X)$ and $(Y, d_Y)$ are assumed to be unbounded proper metric spaces
with base points $x_0$ and $y_0$, respectively.

Let $C(X)$ denote the family of all real-valued continuous functions on $X$,
and let $C^{\ast}(X)$ denote the subfamily consisting of all bounded functions of $C(X)$.

Let $\mathbb{R}_{\infty}$ be the set $[-\infty, \infty]$ endowed with the order topology.
For a lattice $\LL\subset C(X)$, we consider its evaluation map
\[e_{\LL}: X\to \mathbb{R}_{\infty}^{\LL},\]
defined by $e_{\LL}(x)=(f(x))_{f\in \LL}$ for every $x\in X$.
A unital vector lattice $\LL\subset C(X)$ is said to \textit{separate points and closed sets} in $X$ if, for each closed set $F\subset X$ and any point $p\in X\setminus F$,
there exists $f\in \LL$ such that $f(p)\not\in \mathrm{cl}_{\mathbb{R}}\, f(F)$.
If $\LL$ separates points and closed sets in $X$,
then $e_{\LL}$ is a topological embedding \cite[2.3.20]{Eng} (cf. \cite[p.14]{LT}).
Identifying $X$ with $e_{\LL^{\ast}}(X)$,
the closure
\[K(\LL)=\overline{e_{\LL}(X)}\]
in $\mathbb{R}_{\infty}^{\LL}$
gives a compactification of $X$.

Let $\K(X)$ denote the set of all compactifications of $X$.
For $\alpha X, \gamma X \in \K(X)$,
we say $\alpha X \succeq \gamma X$ provided that
there is a continuous map $f:\alpha X \to \gamma X$ such that $f|_{X}=\mathrm{id}_X$.
If $\alpha X \preceq  \gamma X$ and $\alpha X \succeq \gamma X$,
then we say that $\alpha X$ and $\gamma X$ are \textit{equivalent compactifications} of $X$.
Two equivalent compactifications of $X$ are homeomorphic.
We identify equivalent elements of $\K(X)$,
 and then $(\K(X), \preceq)$ becomes an ordered set.

For a lattice $\LL$, $\LL^{\ast}$ represents its sublattice of all bounded functions.
 Abusing notation, we allow the use of symbols like $\LL^{\ast}$ to represent a lattice consisting of bounded functions.
A unital vector lattice $\LL^{\ast} \subset C^{\ast}(X)$
is called a \textit{complete ring of functions} on $X$
if  $\LL^{\ast}$ contains all constant maps, separates points from closed sets,
and is a closed subring of $C^{\ast} (X)$ with respect to the sup-metric \cite[3.12.22(e)]{Eng}.

Let $\C(X)$ denote the family of all complete rings of functions on $X$.
For each $\gamma X\in \K(X)$, define 
\[S(\gamma X)=\{f|_{X} : f\in C(\gamma X)\}.\]
For a Tychonoff space $X$, it is well-known that
 $S$ and $K$ define order isomorphisms
$S: (\K(X), \preceq) \to (\C(X), \subset)$
 and $K: (\C(X), \subset)\to (\K(X), \preceq)$ such that
 $S\circ K=\textrm{id}_{\C(X)}$ and $K\circ S= \textrm{id}_{\K(X)}$.
Thus, we obtain the following fundamental fact (cf. \cite[3.12.22]{Eng}, \cite[4.5]{PW}).

 \begin{proposition}\label{PROP_2-1}
  If $\LL^{\ast}\subset C^{\ast}(X)$ is a complete ring of functions on a Tychonoff space $X$,
   then $\LL^{\ast}$ coincides with the family of all continuous functions on $X$
   that are continuously extendable over the compactification $K(\LL^{\ast})$ of $X$.
   \hfill$\square$
 \end{proposition}

Let $(X,d_X )$ and $(Y,d_Y)$ be proper metric spaces.
A function $f: X\to Y$
is said to be \textit{slowly oscillating}
provided that, given $R>0$
and $\varepsilon>0$, there exists a compact subset $K\subset X$
such that
\[\diam_{d_Y} f(B_{d_X}(x,R))<\varepsilon\]
for every $x\in X\setminus K$,
where $\diam_{d_Y} A=\sup \{ d_Y(x,y): x,y\in A\},~A\subset Y $.
Although one often considers slowly oscillating functions without continuity in the context of coarse geometry,
we will only deal with continuous functions.

Let $\SO(X)$ denote the set of all real-valued slowly oscillating continuous functions on a proper metric space $(X, d_X)$,
and let $\SO^{\ast}(X)$ denote the subset consisting of all bounded functions in $\SO(X)$.
Let $f, g : X \to \R$ be slowly oscillating continuous functions.
We define $f\wedge g$ and $f\vee g$ as functions
such that $(f\wedge g)(x)=\min\{f(x), g(x)\}$
and $(f\vee g)(x)=\max\{f(x), g(x)\}$ for all $x\in X$.

\begin{lemma}
    Both $\SO(X)$ and $\SO^{\ast}(X)$ are lattices.
\end{lemma}

\begin{proof}
    We only show that $\SO(X)$ becomes a lattice.
The proof for $\SO^{\ast}(X)$ is similar and is left to the reader.
\par
    Let $f, g\in \SO(X)$.
We shall check that $f\wedge g\in \SO(X)$.
Given $\varepsilon > 0$ and $R>0$, let $\delta =\varepsilon/3$.
Take a compact subset $K\subset X$ such that
\[ \max\left\{ \diam f(B_{d_X}(x,R)),~\diam g(B_{d_X}(x,R)) \right\} <\delta\]
for every $x\in X\setminus K$.
Suppose that $x\in X\setminus K$ and $a, b\in B_{d_X}(x, R)$ are points such that
$(f\wedge g)(a)=f(a)$ and $(f\wedge g)(b)=g(b)$,
i.e., $f(a)\leq g(a)$ and $g(b)\leq f(b)$.
Then we have
\begin{align*}
  |(f\wedge g)(a)-(f\wedge g)(b)|
  &=|f(a) -g(b) |\\
 &\leq |f(a) -f(b)|+|f(b)-g(b)|\\
 &\leq \delta + f(b)-g(b)\\
 &\leq \delta + f(a)+\delta -g(b)\\
 &\leq 2\delta + g(a)-g(b) \\
 &\leq 3\delta=\varepsilon.
\end{align*}
It can be shown in a similar way for other cases, so it follows that $f \wedge g \in \SO(X)$. 
In the same manner, we have $f \vee g \in \SO(X)$, and thus, $\SO(X)$ is a lattice.
\end{proof}

Let $\U(X)$ denote the lattice of all real-valued uniformly continuous functions
on a proper metric space $(X, d_X)$, and let $\U^{\ast}(X)$ denote the sublattice consisting of all bounded functions in $\U(X)$.
Note that $\SO(X)$ is a sublattice of $\U(X)$ whenever $d_X$ is proper (cf. \cite{Iwa}).
Since $\U^{\ast}(X)$ and $\SO^{\ast}(X)$ are complete rings of functions on $X$,
 they determine the
compactifications $sX=K(\U^{\ast}(X))$ and $hX=K(\SO^{\ast}(X))$
known as the \textit{Samuel-Smirnov compactification} (cf. \cite{Woods})
and the \textit{Higson compactification} of $X$, respectively.
In particular,
the remainder $\nu X=hX \setminus X$ is called the \textit{Higson corona} of $X$
(cf. \cite{Keesling}, \cite{Roe}).
Since $\SO^{\ast}(X)$ is a closed subring of $\U^{\ast}(X)$,
we have $hX\preceq sX\preceq \beta X$,
where $\beta X$ denotes the Stone-\v{C}ech compactification of $X$.

\section{Homeomorphisms Induced by Lattice Isomorphisms}
In this section, we verify that the lattices $\SO(X)$ and $\SO^{\ast}(X)$,
 defined on a proper metric space $X$,
 satisfy the conditions $(\spadesuit)$ and $(\heartsuit)$ stated in \cite{LT}.
Consequently, we obtain homeomorphisms induced by lattice isomorphisms between lattices of slowly oscillating functions.
\par
Let $\LL$ and $\MM$ be lattices.
A map $T :\LL \to \MM$ is called a \textit{lattice homomorphism}
if it preserves joins and meets, that is,
\[T(f\vee g)=T(f)\vee T(g)~\text{and}~T(f\wedge g)=T(f)\wedge T(g)\]
for all $f,g\in \LL$.
A bijective lattice homomorphism is called a \textit{lattice isomorphism}.

A set of points $S$ in a metric space $(X, d_X)$ is {\itshape separated}
 if there exists $\varepsilon >0$ such that $d_X (x_1, x_2)>\varepsilon$
  whenever $x_1$ and $x_2$ are distinct points in $S$.

The following is a restriction of Proposition 4.1 in \cite{LT} to vector lattices.

\begin{proposition}\label{PROP_3-1}
  Let $A(X)$ and $A(Y)$ be vector latices defined on metric spaces $X$ and $Y$ respectively.
  Assume the following condition:
  \begin{enumerate}
    \item[$(\spadesuit 1)$]
    $Y$ is complete, and for any separated sequences $(y_n)$ in $Y$,
    there exists $g\in A(Y)$ such that the sequence $(g(y_n))$ diverges in $\mathbb{R}$.
  \end{enumerate}
  If $\varphi : K(A(X))\to K(A(Y))$ is a homeomorphism, then $\varphi (X)\subset Y$.
  \hfill$\square$
\end{proposition}

For a vector subspace $A(Y)$ of $C(Y)$, consider the following condition:
\begin{enumerate}
  \item[$(\spadesuit 2)$]
  If $g\in A(Y)$ and $h\in C^{\infty}(\mathbb{R})$ such that 
  \[||h^{(k)}||_{\infty}
  =\sup \{|h^{(k)}(x)| : x\in \mathbb{R}\}
  <\infty\] for all $k\geq 1$,
  then $h\circ g \in A(Y)$.
\end{enumerate}
We say that $A(Y)$ satisfies $(\spadesuit)$ if it satisfies both $(\spadesuit 1)$ and $(\spadesuit 2)$.

The following is a restriction of Theorem 4.3 in \cite{LT} to vector lattices satisfying condition $(\spadesuit)$.
\begin{thm}\label{THM_3-2}
  Let $A(X)$ and $A(Y)$ be vector lattices on metric spaces $X$ and $Y$ respectively
  and let $T: A(X)\to A(Y)$ be a lattice isomorphism.
  Assume that both $A(X)$ and $A(Y)$ satisfies $(\spadesuit)$.
  Then there is a homeomorphism $\tau: X\to Y$
  such that for any $f,g\in A(X)$ and any open set $U$ in $X$,
   $f\geq g $ on $U$ if and only if $T(f) \geq T(g)$ on $\tau (U)$.
   \hfill$\square$
\end{thm}

If a homeomorphism $\tau : X\to Y$ satisfies the condition stated in Theorem \ref{THM_3-2},
 then it is called a homeomorphism {\itshape induced by} $T$.

For $f, g\in C(X)$,
 let $\{f<g\}=\{x\in X : f(x)<g(x)\}$.
Consider the following property of a vector subspace $A(X)$ of $C(X)$ at a point $x\in X$:
\begin{enumerate}
  \item[$(\heartsuit_x)$]
  Either $x$ is an isolated point of $X$,
   or if $f\in A(X)$, $f\geq 0$ and $f(x)=0$,
   then there exists $g\in A(X)$ such that $x\in \overline{\{f<g\}}\cap \overline{\{g<0\}}$.
\end{enumerate}
We say that $A(X)$ satisfies $(\heartsuit)$ if it satisfies condition $(\heartsuit_x)$ for every $x\in X$.

The following is a direct consequence of Theorem 4.5 in \cite{LT}.

\begin{thm}\label{THM_3-3}
  Let $A(X)$ and $A(Y)$ be vector lattices on metric spaces $X$ and $Y$ respectively
  and let $T: A(X)\to A(Y)$ be a lattice isomorphism.
  Assume that both $A(X)$ and $A(Y)$ satisfies $(\spadesuit)$ and $(\heartsuit)$.
  Then there exist an homeomorphism $\tau :X \to Y$ induced by $T$
   and a correspondence $\mathfrak{t}: Y\times \mathbb{R}\to\mathbb{R}$
   such that 
   \begin{enumerate}
    \item $\mathfrak{t} (y, \cdot) : \mathbb{R}\to\mathbb{R}$
     is an increasing homeomorphism for each $y\in Y$, and
    \item $T(f)(y)=\mathfrak{t} (y, f(\tau^{-1}(y)))$ for every $f\in A(X)$ and $y\in Y$.
    \hfill$\square$
   \end{enumerate}
 
\end{thm}

\begin{remark}\label{REM_3-4}
  In the setting of Theorem \ref{THM_3-3},
  it follows that
  for any $f,g\in A(X)$ and any $y\in Y$,
  $T(f)(y)=T(g)(y)$ if and only if $f(\tau^{-1}(y))=g(\tau^{-1}(y))$.
  \[
    \begin{CD}
      \R @. \R \\
      @A{f(\tau^{-1}(y))=g(\tau^{-1}(y))}AA   @AA{T(f)(y)=T(g)(y)}A \\
      X   @>{\tau}>> Y
    \end{CD}
  \]  
Additionally, consider the following condition:
\begin{enumerate}
  \item[$(\dag)$] For any two distinct points $x_1, x_2 \in X$,
  there exist $f, g \in A(X)$ such that $f(x_1)=g(x_1)$ while $f(x_2)\neq g(x_2)$.
\end{enumerate}
If $A(X)$ satisfies condition $(\dag)$,
 then a homeomorphism $\tau : X\to Y$ induced by $T$ in Theorem \ref{THM_3-3} is uniquely determined.
Indeed, suppose that there exists another homeomorphism $\varphi :X \to Y$ induced by $T$
 such that $\varphi \neq \tau$.
Let $y\in Y$ be a point such that $\tau^{-1}(y)\neq \varphi^{-1}(y)$.
By $(\dag)$, there exist $f, g \in A(X)$
 such that $f(\tau^{-1}(y))=g(\tau^{-1}(y))$ while $f(\varphi^{-1}(y))\neq g(\varphi^{-1}(y))$.
Then $f(\tau^{-1}(y))=g(\tau^{-1}(y))$ implies $T(f)(y)=T(g)(y)$
whereas $f(\varphi^{-1}(y))\neq g(\varphi^{-1}(y))$ implies 
$T(f)(y)\neq T(g)(y)$,
a contradiction.
\hfill$\square$
\end{remark}

To simplify notation, we represent $d_X(x_0, x)$ as $|x|$ for each $x \in X$.

\begin{proposition}\label{PROP_3-5}
  For a proper metric space $X$, both $\SO(X)$ and $\SO^{\ast}(X)$ satisfy
   conditions $(\dag)$, $(\spadesuit)$ and $(\heartsuit)$.
\end{proposition}

\begin{proof}
Condition $(\dag)$ is obvious.
To verify condition $(\spadesuit 1)$, let $(x_n) \subset X$ be a separated sequence.  
Since $X$ is a proper metric space, it follows that $|x_n| \to \infty$ as $n \to \infty$.  
By taking a subsequence if necessary, we may assume the following:  
\[|x_{2n}| + 2n^2 < |x_{2n+1}| < |x_{2n+2}| - 2(n+1)^2\] for every $n$.  
Then we can take $f \in \SO^{\ast}(X) \subset \SO(X)$ such that $f(x_{2n}) = 0$  
and $f(x_{2n+1}) = 1$ for every $n$.  
Indeed, define $f: X \to [0, 1]$ as follows:  
\[
f(x) =  
\begin{cases}  
1 - d_X(x, x_{2n+1})/n, & \text{if } x \in B_{d_X}(x_{2n+1}, n), \\  
0, & \text{if } x \notin \bigcup_{n=1}^\infty B_{d_X}(x_{2n+1}, n).  
\end{cases}  
\]

If $f: X \to \mathbb{R}$ is slowly oscillating and $h \in C^\infty(\mathbb{R})$ is a function 
 whose derivative $h'$ satisfies $\|h'\|_\infty < \infty$,  
then $h \circ f: X \to \mathbb{R}$ is also slowly oscillating.  
Therefore, condition $(\spadesuit 2)$ holds.  

Since $(\heartsuit_x)$ is a local condition, it is straightforward to verify that $(\heartsuit_x)$ holds for every $x \in X$  
(cf. Example C (c) in \cite{LT}).  
\end{proof}

The following result follows from Theorem \ref{THM_3-3}, Remark \ref{REM_3-4}, and Proposition \ref{PROP_3-5}:

 \begin{thm}\label{THM_3-6}
  Let $(X, d_X)$ and $(Y, d_Y)$ be proper metric spaces.
  Let $T: \SO(X) \to \SO(Y)$ be a lattice isomorphism.
  Then there exist a uniquely determined homeomorphism $\tau :X \to Y$ induced by $T$
  and a correspondence $\mathfrak{t}: Y\times \mathbb{R}\to\mathbb{R}$
  such that 
  \begin{enumerate}
   \item $\mathfrak{t} (y, \cdot) : \mathbb{R}\to\mathbb{R}$ is an increasing homeomorphism for each $y\in Y$, and
   \item $T(f)(y)=\mathfrak{t} (y, f(\tau^{-1}(y)))$ for every $f\in A(X)$ and $y\in Y$.
   \end{enumerate}
   The same holds for a lattice isomorphism $T: \SO^{\ast}(X) \to \SO^{\ast}(Y)$.
  \hfill$\square$
 \end{thm}

 In what follows, we will refer to the uniquely determined homeomorphism in Theorem \ref{THM_3-6}
  as {\itshape the} homeomorphism induced by $T$.

  \begin{remark}
    Another construction of the induced homeomorphisms can be described in terms of correspondence of regular open sets.  
    For further details, see \cite{FJ_Cabello_Sanchez}.  
    \hfill$\square$
  \end{remark}

For a continuous map $h : X\to Y$,
let \[h^{\ast}: C(Y)\to C(X)\]
be the composition operator,
that is, $h^{\ast}(f)=f\circ h$ for every $f\in C(Y)$.
Then $h^{\ast}$ is a lattice homomorphism.
If $h$ is a homeomorphism, then $h^{\ast}$ becomes a lattice isomorphism.

For linear lattice isomorphisms,
Theorem \ref{THM_3-6} yields the following:

\begin{thm}\label{THM_3-8}
  Let $(X, d_X)$ and $(Y, d_Y)$ be proper metric spaces.
  Let $T: \SO(X) \to \SO(Y)$ be a linear lattice isomorphism,
  and let $\tau :X \to Y$ be the homeomorphism induced by $T$.
  Then $T$ can be expressed as $T=\omega \cdot (\tau^{-1})^{\ast}$, where $\omega=T(1)$.
  The same representation holds for a linear lattice isomorphism $T: \SO^{\ast}(X) \to \SO^{\ast}(Y)$.
\end{thm}

\begin{proof}
  Let $\omega =T(1)\in \SO(Y)$.
  For $y\in Y$ and $f\in \SO(X)$,
  let $c=f(\tau^{-1}(y))$.
  Then we have $T(f)(y)=T(c)(y)=c\cdot T(1)(y)=f(\tau^{-1}(y))\cdot \omega(y)$.
\end{proof}

\section{Isomorphisms on Lattices of Slowly Oscillating Functions}

Let $f:X\to Y$ be a map between metric spaces $(X, d_X)$ and $(Y, d_Y)$.
The map $f$ is \textit{uniformly expansive} (or \textit{uniformly bornologous}) if there exists a non-decreasing function $\sigma : [0, \infty)\to [0, \infty)$ satisfying
\[
  d_Y (f(x), f(x'))\leq \sigma (d_X (x, x'))
\]
for every $x, x' \in X$.
The map $f$ is \textit{metrically proper} if,
for any bounded subset $B \subset Y$, $f^{-1}(B)$ is bounded in $X$.
Note that a homeomorphism between proper metric spaces is metrically proper.
A map that is both uniformly expansive and metrically proper is called a \textit{coarse map}.
Two maps $f, g:X\to Y$ are said to be \textit{close} if there exists $r>0$
such that $d_Y (f(x), g(x))<r$ for every $x\in X$.
A coarse map $f: X\to Y$ is called a \textit{coarse equivalence}
if there exists a coarse map $g:Y\to X$ such that
$g\circ f$ and $f\circ g$ are close to $\mathrm{id}_X$ and $\mathrm{id}_Y$, respectively.
If there exists a coarse equivalence between $X$ and $Y$, then
$X$ and $Y$ are said to be \textit{coarsely equivalent}.
Coarsely equivalent proper metric spaces have homeomorphic Higson coronas \cite[Corollary 2.42]{Roe}.
A homeomorphism $f:X\to Y$ is called a \textit{coarse homeomorphism}
if both $f$ and its inverse $f^{-1}$ are coarse maps.

The following is a fundamental relation between coarse maps and slowly oscillating functions.

\begin{lemma}\label{LEM_4-1}
  If $h:X \to Y$ is a coarse continuous map between proper metric spaces
  and $f: Y\to \R$ is slowly oscillating
  then $f\circ h: X\to \R$ is a slowly oscillating function.
\end{lemma}

\begin{proof}
Let $\varepsilon >0$ and $R>0$.
Since $h$ is uniformly expansive,
    there exists a non-decreasing function $\sigma :  [0, \infty)\to [0, \infty)$
    such that
\[ d_Y (h(x), h(x'))\leq \sigma (d_X (x, x')) \]
for every $x, x' \in X$.
Since $f$ is slowly oscillating,
 we can find a compact subset $L\subset Y$ such that
\[\diam f(B_{d_Y}(y, \sigma (R))) <\varepsilon\]
for every $y\in Y\setminus L$.
Put $K=h^{-1}(L)$.
Then $K$ is compact since $h$ is a metrically proper map between proper metric spaces.
Now, for any $x\in X\setminus K$,
 we have $h(x)\in Y\setminus L$ and 
 \[h(B_{d_X}(x, R))\subset B_{d_Y}(h(x), \sigma (R)). \]
This implies that
 \[ \diam f\circ h(B_{d_X}(x, R)) \leq \diam f (B_{d_Y}(h(x), \sigma (R)))<\varepsilon.\]
Therefore, $f\circ h$ is slowly oscillating.
\end{proof}

\subsection{Nonlinear lattice isomorphisms}

Let $r>0$.
Recall that a metric space $X$ is said to be $r$-chain connected
if, for every two points $x$ and $x'$ of $X$,
there exist finite points $p_0, \dots, p_n$ in $X$ such that
$p_0 =x$, $p_n =x'$ and $d_X(p_{i-1}, p_{i})\leq r$ for every $i = 1,\dots, n$.
A space is said to be chain-connected if it is $r$-chain connected for some $r > 0$.

\begin{lemma}\label{LEM_4-2} 
  Let $(X, d_X)$ and $(Y, d_Y)$ be chain-connected proper metric spaces.
  Suppose $T: \SO(X) \to \SO(Y)$ is a lattice isomorphism with $T(0) = 0$.
  Then, the homeomorphism $\tau: X \to Y$ induced by $T$ is uniformly expansive.
  The same conclusion holds for lattice isomorphisms $T: \SO^{\ast}(X) \to \SO^{\ast}(Y)$
   with $T(0)=0$. 
\end{lemma}

\begin{proof}
  We prove the lemma only for $T: \SO(X) \to \SO(Y)$. 
  The proof for $T: \SO^{\ast}(X) \to \SO^{\ast}(Y)$ can be shown in a similar way.

  Suppose on the contrary that $\tau$ is not uniformly expansive.
  Then we can take $R>0$ and two sequences $(a_n)$ and $(b_n)$ in $X$ such that
  \begin{enumerate}
    \item\label{4.3-1} $0 < d_X(a_n, b_n) < R$ and
    \item\label{4.3-2} $d_Y (\tau(a_n), \tau (b_n)) > n$
  \end{enumerate}
  for every $n$.
  Since $d_X$ is a proper metric, (\ref{4.3-1}) and (\ref{4.3-2}) imply that
  \begin{enumerate}\setcounter{enumi}{2}
    \item\label{4.3-3} $|a_n| \to \infty$, $|b_n| \to \infty$ ($n \to \infty$).
  \end{enumerate}
  Since $\tau$ is metrically proper, we have
  \begin{enumerate}\setcounter{enumi}{3}
    \item\label{4.3-4} $|\tau(a_n)| \to \infty$, $|\tau(b_n)| \to \infty$ ($n \to \infty$).
  \end{enumerate}
  Then, taking subsequences if necessary, we may assume that
  \begin{enumerate}
    \setcounter{enumi}{4}
    \item\label{4.3-5} $|\tau(a_n)| + 2n^2 < |\tau(b_n)| < |\tau(a_{n+1})| - 2(n+1)^2$ for every $n$.
  \end{enumerate}
  Thus,
  we can take $g \in \SO^{\ast}(Y)$ such that $g(\tau(a_n)) = 0$ and $g(\tau(b_n)) = 1$.
  Put $f = T^{-1}(g) \in \SO (X)$.
  By Theorem \ref{THM_3-6}, $g(\tau(a_n)) = 0$ implies that
   $f(a_n) = T^{-1}(g) (a_n) = T^{-1}(0)(a_n) = 0$.
  Consequently, (\ref{4.3-1}) and (\ref{4.3-3}) imply $f(b_n) \to 0~ (n \to \infty)$ because $f$ is slowly oscillating.
  Since $Y$ is chain-connected,
  we can take a sequence $(c_n) \subset X$ and $\delta>0$ such that,  for every $n$,
  \begin{enumerate}
    \setcounter{enumi}{5}
    \item\label{4.3-6} $0 < d_Y(\tau(b_n), \tau(c_n)) < \delta$.
  \end{enumerate}
  In particular, we may assume that
  \begin{enumerate}\setcounter{enumi}{6}
    \item $\tau(a_m) \neq \tau(c_n) \neq \tau(b_m)$ for every $m, n$.
  \end{enumerate}
  Since $f(b_n) \to 0 ~ (n\to \infty)$,
  there exists $h \in \SO^{\ast}(X)$ such that $h(a_n) = h(c_n) = 0$ and $h(b_n) = f(b_n)$ for every $n$.
  By Theorem \ref{THM_3-6}, the condition $h(c_n) = 0$ implies
   $T(h)(\tau(c_n)) = 0$ for every $n$.
  Similarly, the condition $h(b_n) = f(b_n)$ implies $T(h) (\tau (b_n)) = T(f)(\tau (b_n))
    = g(\tau (b_n)) = 1$ for every $n$.
  However, (\ref{4.3-4}) and (\ref{4.3-6}) yield that these two properties are not compatible since $T(h)$ is slowly oscillating, a contradiction.
\end{proof}

\begin{thm}\label{THM_4-3}
  Let $(X, d_X)$ and $(Y, d_Y)$ be chain-connected proper metric spaces.
  Suppose $T: \SO(X) \to \SO(Y)$ be a lattice isomorphism with $T(0) = 0$.
  Then, the homeomorphism $\tau :X \to Y$ induced by $T$ is a coarse homeomorphism.
  The same conclusion holds for lattice isomorphisms $T: \SO^{\ast}(X) \to \SO^{\ast}(Y)$
   with $T(0)=0$. 
\end{thm}

\begin{proof}
  Since a homeomorphism between proper metric spaces is metrically proper,
  $\tau$ is a coarse homeomorphism by Lemma \ref{LEM_4-2}.
\end{proof}

For each $f\in C(X)$,
 let $f^{+}$ (resp. $f^{-}$) denote
the non-negative part (resp. the non-positive part) of $f$,
  i.e., $f^{+}=f\vee 0$ and $f^{-}=f\wedge 0$.

The proof of the following result is essentially the same as in \cite[Corollary 2]{FJ_Cabello_Sanchez},
 but we provide it here to clarify the necessity of chain-connectedness and the role of slowly oscillating functions.

\begin{lemma}\label{LEM_4-4}
  Let $(X, d_X)$ and $(Y, d_Y)$ be chain-connected proper metric spaces.
  Let $T: \SO(X) \to \SO(Y)$ be a lattice isomorphism with $T(0)=0$.
  Then $T$ induces the lattice isomorphism $T|_{\SO^{\ast}(X)} :\SO^{\ast}(X)\to\SO^{\ast}(Y)$.
\end{lemma}

\begin{proof}
  We show that $T(\SO^{\ast}(X)) \subset \SO^{\ast}(Y)$.
  Then, applying the assertion to $T^{-1}$ implies
  $T^{-1}(\SO^{\ast}(Y)) \subset \SO^{\ast}(X)$, i.e., $\SO^{\ast}(Y)\subset T(\SO^{\ast}(X))$.
  \par
  Suppose on the contrary that there is $f\in \SO^{\ast}(X)$
  such that $T(f)$ is unbounded.
  As $T(f^{+}) = T(f)^{+}$, $T(f^{-}) = T(f)^{-}$ and $T(f) = T(f)^{+} + T(f)^{-}$,
  either $T(f^{+})$ or $T(f^{-})$ is unbounded.
  Hence, we may assume that $f \geq 0$ without loss of generality.
  Since $f$ is bounded, we can take $a \in \R$ such that $0 \leq f \leq a$.
  Then the function $T(a)$ is unbounded since $0 \leq T(f)\leq T(a)$.
  Consider \[I=\{t \in [0, a] : T(t) \in \SO^{\ast}(X)\}.\]
  Note that $I$ can be expressed as $I = [0, c]$ or $I = [0, c)$ for some $c \geq 0$.
  \par
 We only show the case that $I=[0, c)$ for some $0 < c < a$
  and the case that $I = [0, c]$ is left to the reader (see \cite[Corollary 2]{FJ_Cabello_Sanchez}).
  Let $(c_i) \subset \R$ be an increasing sequence such that
  $0< c_i < c_{i+1}<c$ and $c_i \to c$ $(i \to \infty)$.
  Then $T(c_i)$ is bounded and $T(c)$ is unbounded.
  Put $g_0 = T(c)$.
  For each $i$, put $g_i = T(c_i)$ and let $M_i = \sup\{g_i(y) : y \in Y \}$.
  Obviously, $M_i \leq M_{i+1} < \infty$ for every $i$.
  Since $g_0$ is unbounded, we can take a sequence $(y_i) \subset Y$
  such that, for every $i$,
  \[ g_0(y_i) > \max\{M_i +i,~ g_0 (y_{i-1}) + 1\},\]
  where $y_0$ is the base point of $X$.
  Since $|g_0 (y_i)| \to \infty$, it follows that $|y_i| \to \infty$ $(i \to \infty)$.
  We may assume that $|y_i| < |y_{i+1}|$ for every $i$.
  The condition
  $g_0 (y_{i+1}) - g_0 (y_i) > 1$ implies that
  $d(y_i, y_{i+1}) \to \infty$
  since $g_0$ is slowly oscillating.
  The chain-connectedness of $X$ assures that there exist $\delta > 0$
  and a sequence $(y_i') \subset X$ such that
  $d(y_i, y_i') < \delta$ for every $i$.
  In particular, we may assume that $y_i \neq y_j'$
  for every $i, j$.
  Since the sequence $(c_i)$ converges to $c$,
  we can take $h \in \SO^{\ast}(Y)$ such that
  $h(y_i) = c_i$ and $h(y_i') = c$.
  Put $\xi = h \circ \tau \in \SO(X)$.
  As $\xi (\tau^{-1}(y_i)) = h(y_i) = c_i$ and $\xi (\tau^{-1}(y_i')) = h(y_i') = c$,
  we have
  $T(\xi)(y_i) = T(c_i)(y_i) = g_i (y_i) \leq M_i$
  and
  $T(\xi)(y_i') = T(c)(y_i') = g_0 (y_i')$.
  Since $g_0 (y_i)>M_i +i$ and $g_0$ is slowly oscillating,
  $g_0 (y_i') \geq  M_i + i-1$ for sufficiently large $i$.
  Consequently,  for sufficiently large $i$, we have
  $| T(\xi)(y_i) - T(\xi)(y_i')|
  \geq i - 1 \to \infty$.
  This is impossible because $T(\xi)$ is slowly oscillating, a contradiction.
\end{proof}

Now we obtain the following Banach-Stone-like theorem for lattices of slowly oscillating functions.
For results concerning lattices of uniformly continuous functions, see \cite{FJ_Cabello_Sanchez}, \cite{Garrido-Jaramillo-2000}, \cite{Husek}, \cite{H-P-1}, and \cite{H-P-2}.

\begin{thm}[Banach-Stone-like theorem]\label{THM_4-5}
  Let $(X, d_X)$ and $(Y, d_Y)$ be  chain-connected proper metric spaces.
  The following are equivalent:
  \begin{enumerate}
    \item\label{20-1} $\SO(X)$  and $\SO(Y)$ are lattice isomorphic.
    \item\label{20-2} $\SO^{\ast}(X)$  and $\SO^{\ast}(Y)$ are lattice isomorphic.
    \item\label{20-3} There exists a coarse homeomorphism between $X$ and $Y$.
    \item\label{20-4} Higson compactifications $hX$ and $hY$ are homeomorphic.
  \end{enumerate}
\end{thm}

\begin{proof}
  We shall show that $(\ref{20-1}) \Leftrightarrow (\ref{20-2}) \Leftrightarrow (\ref{20-3})$.
  Note that every lattice isomorphism $T$ can be adjusted to satisfy $T(0)=0$.
  By Lemma \ref{LEM_4-4}, $(\ref{20-1})$ implies $(\ref{20-2})$.
  The implication $(\ref{20-2}) \Rightarrow (\ref{20-3})$ follows from Theorem \ref{THM_4-3}.
  Lemma \ref{LEM_4-1} yields that
  if $\varphi : X \to Y$ is a coarse homeomorphism, then
  $\varphi^{\ast}: C(Y) \to C(X)$ induces a lattice isomorphism
  $\varphi^{\ast}|_{\SO(Y)}: \SO(Y) \to \SO(X)$,
  i.e., $(\ref{20-3}) \Rightarrow (\ref{20-1})$.
  \par
  Suppose that (\ref{20-3}) holds, i.e.,
  there exists a coarse homeomorphism $\varphi: X \to Y$.
  Let $e_{\SO^{\ast}(Y)} : Y \to K(\SO^{\ast}(Y)) = hY$ be the evaluation map.
  By Proposition \ref{PROP_2-1} and Lemma \ref{LEM_4-1}, the composition
  $e_{\SO^{\ast}(Y)}\circ \varphi : X \to hY$ is uniquely extended to a continuous map
  $\hat{\varphi}: hX \to hY$.
  \[
    \begin{CD}
      hX @>{\hat{\varphi}}>> hY \\
      @A{e_{\SO^{\ast}(X)}}AA   @AA{e_{\SO^{\ast}(Y)}}A \\
      X   @>{\varphi}>> Y
    \end{CD}
  \]
  Then $\hat{\varphi}: hX \to hY$ must be a homeomorphism
   since $\varphi$ is a homeomorphism, i.e., $(\ref{20-3}) \Rightarrow (\ref{20-4})$.
  \par
  Suppose that (\ref{20-4}) holds, that is,
  there exists a homeomorphism $\hat{\varphi}: hX \to hY$.
  Lemma \ref{PROP_3-1} implies that $\hat{\varphi}(X) = Y$,
  and thus, $\varphi = \hat{\varphi}|_{X}$ is a homeomorphism from $X$ onto $Y$.
  Then $\varphi^{\ast}: C(Y) \to C(X)$ induces a lattice isomorphism
  $\varphi^{\ast}|_{\SO^{\ast}(Y)}: \SO^{\ast}(Y) \to \SO^{\ast}(X)$.
  Indeed, for $f \in  \SO^{\ast}(Y)$,
  $f\circ \varphi : X \to \R$ has the extension $\hat{f} \circ \hat{\varphi} :hX \to \R$,
  where $\hat{f}: hY \to \R$ is the extension of $f$ whose existence is guaranteed by Proposition \ref{PROP_2-1}.
  Thus, $f \circ \varphi$ is slowly oscillating by Proposition \ref{PROP_2-1}.
  Since $\varphi$ is a homeomorphism, $\varphi^{\ast}|_{\SO^{\ast}(Y)}$ must be a lattice isomorphism,
  i.e., $(\ref{20-4}) \Rightarrow (\ref{20-2})$.
\end{proof}

\begin{remark}
  In the proof of the above theorem, it is shown that if $hX$ and $ hY $ are coarsely homeomorphic, then their Higson coronas $ \nu X $ and $ \nu Y $ are homeomorphic.
  However, the fact that $ \nu X $ and $ \nu Y $ are homeomorphic does not imply that $X$ and $Y$ are coarsely homeomorphic in general.
 As mentioned in the Introduction, any unbounded chain-connected proper metric space has positive asymptotic dimension. 
 Thus, this raises the following question:
 If $X$ and $Y$ are proper metric spaces with positive asymptotic dimension
  and their Higson coronas $\nu X$ and $\nu Y$ are homeomorphic, does it follow that $X$ and $Y$ are coarsely equivalent?
  \hfill$\square$
\end{remark}

\subsection{Linear lattice isomorphisms}

Next, we shall consider a linear lattice isomorphism $T:\LL\to \MM$, i.e.,
$T$ is a bijection such that
\begin{enumerate}
  \item[(i)]  $T(f \vee g) = T(f) \vee T(g)$, $T(f\wedge g) = T(f) \wedge T(g)$ and
  \item[(ii)] $T(kf + lg) = kT(f) + lT(g)$
\end{enumerate}
for every $f,g \in \LL$ and $k, l \in \R$.
Note that $T(0) = 0$ and $T(|f|) = |T(f)|$ for every $f \in \LL$.

Let $\SF(X)$
be the subset of $\SO(X)$ such that $\omega \in \SF(X)$
if and only if $\omega$ satisfies the conditions:
\begin{enumerate}
  \item[$(\sharp 1)$] $\omega > 0$,
  \item[$(\sharp 2)$] $\omega \cdot f \in \SO(X)$ and
  \item[$(\sharp 3)$] $(1/\omega) \cdot f \in \SO(X)$
\end{enumerate}
for every $f \in \SO(X)$.
Note that if $\omega \in \SF(X)$ then $1/ \omega \in \SF(X)$.
Let $\SF^{\ast}(X)$ be the set of bounded functions $\omega \in \SO^{\ast}(X)$
satisfying conditions $(\sharp 1)$, $(\sharp 2)$ and $(\sharp 3)$,
 where $\SO(X)$ is replaced by $\SO^{\ast}(X)$.

\begin{lemma}\label{LEM_4-7}
  Let $(X, d_X)$ be a proper metric space.
  The sets $\SF(X)$ and $\SF^{\ast}(X)$ are sublattices of $\SO(X)$ and $\SO^{\ast}(X)$, respectively.
\end{lemma}

\begin{proof}
  We will show that $\SF(X)$ is a sublattice of $\SO(X)$.
  To do this, let $\omega, \lambda \in \SF(X)$.
  Since $\omega>0$ and $\lambda>0$,
  we have
  $\omega \wedge \lambda>0$, $\omega \vee \lambda>0$,
  i.e., both $\omega \wedge \lambda$ and $\omega \vee \lambda$ satisfy $(\sharp 1)$.

  Let $f\in \SO(X)$.
  Then we have 
  \[(\omega \wedge \lambda) \cdot f^{+}
    =(\omega \cdot f^{+}) \wedge (\lambda \cdot f^{+}) \in \SO(X)\]
  and
  \[(\omega \wedge \lambda) \cdot f^{-}
    =(\omega \cdot f^{-}) \vee (\lambda \cdot f^{-}) \in \SO(X),\]
  i.e.,
  \[(\omega \wedge \lambda) \cdot f = (\omega \wedge \lambda) \cdot (f^{+} + f^{-})
    =(\omega \wedge \lambda) \cdot f^{+} + (\omega \wedge \lambda) \cdot f^{-} \in \SO(X).\]
  Similarly, we obtain $(\omega \vee \lambda) \cdot f \in \SO(X)$.
  Thus, both $\omega \wedge \lambda$ and $\omega \vee \lambda$ satisfy condition $(\sharp 2)$.
\par
  Since $1/ \omega , 1/ \lambda \in \SF(X) $,
   it follows from the above discussion that both
  $(1/ \omega ) \vee (1/ \lambda)$
  and
  $(1/ \omega ) \wedge (1/ \lambda)$
  satisfy condition $(\sharp 2)$.
  Hence, we have \[\frac{1}{\omega \wedge \lambda} \cdot f 
  = \left(\frac{1}{\omega} \vee \frac{1}{\lambda} \right)\cdot f \in \SO(X)\] and
  \[\frac{1}{\omega \vee \lambda} \cdot f
    =\left(\frac{1}{\omega} \wedge \frac{1}{\lambda} \right)\cdot f \in \SO(X),\]
  i.e., both $\omega \wedge \lambda$ and $\omega \vee \lambda$
  satisfy condition $(\sharp 3)$.
  Therefore,
  $\omega \wedge \lambda$ and $\omega \vee \lambda$ are elements of $\SF(X)$,
  which shows that $\SF(X)$ is a sublattice of $\SO(X)$.
  The proof for $\SF^{\ast}(X)$ follows similarly by restricting all functions to be bounded.
\end{proof}

A function $f\in C(X)$ is called \textit{strongly positive}
if there exists $c>0$ such that $c \leq f$.

\begin{lemma}\label{LEM_4-8}
  Let $(X, d_X)$ be a proper metric space.
  The lattice
  $\SF^{\ast}(X)$ consists of all strongly positive elements of $\SO^{\ast}(X)$.
\end{lemma}

\begin{proof}
  Let $\omega \in \SF^{\ast}(X)$.
  Since $1/\omega \in \SF^{\ast}(X)$,
  there exists a positive number $0 < s < \infty$ such that
  $0< 1/\omega \leq s$, i.e., $0<1/s \leq \omega < \infty$.
  Therefore, $\omega$ is a strongly positive bounded function.

  Now, suppose that $\lambda \in \SO^{\ast}(X)$ is a strongly positive function.
  Let $0<s, t < \infty$ be positive numbers such that
  $s \leq \lambda \leq t$.
  Then $1/t \leq 1/\lambda \leq 1/s$, i.e.,
  $1/\lambda$ is also a strongly positive bounded function.
  Hence, for $x, x'\in X$, we have
  \[\frac{|\lambda(x) - \lambda(x')|}{t^2} \leq \left|\frac{1}{\lambda(x)} - \frac{1}{\lambda(x')}\right|
    \leq \frac{|\lambda(x) - \lambda(x')|}{s^2}\]
  which implies that $1/\lambda$ is slowly oscillating, i.e., $1/\lambda \in \SO^{\ast} (X)$.
  Thus,
  both $\lambda \cdot f$ and $(1/\lambda) \cdot f$ are slowly oscillating
  whenever $f \in \SO^{\ast}(X)$,
  i.e., $\lambda \in \SF^{\ast}(X)$.
\end{proof}

\begin{remark}\label{REM_4-9}
  Consider $X = \{n^2 : n \in \N\} \subset \N$ with the inherited metric from $\N$.
  Then $\SO(X) = C(X)$ and $\SF(X) = \{f\in C(X) : f>0\}$.
  Thus, $\SF(X)$ contains unbounded functions.
  In particular, it contains functions that are not strongly positive.
  \hfill$\square$
\end{remark}


\begin{thm}\label{THM_4-10}
  Let $(X, d_X)$ and $(Y, d_Y)$ be chain-connected proper metric spaces.
  Let $T: \SO(X) \to \SO(Y)$ (resp. $T: \SO^{\ast}(X) \to \SO^{\ast}(Y)$) be a linear lattice isomorphism, and let $\omega = T(1)$.
  Then the homeomorphism $\tau : X \to Y$ induced by $T$ is a coarse homeomorphism
  such that $T = \omega \cdot (\tau^{-1})^{\ast}$.
  In particular, $\omega \in \SF(Y)$ (resp.  $\omega \in \SF^{\ast}(Y)$).
\end{thm}

\begin{proof}
    We only show the case for $T: \SO(X) \to \SO(Y)$.
  The proof for the case of $T: \SO^{\ast}(X) \to \SO^{\ast}(Y)$ is similar and is left to the reader.

  The first part of the theorem follows from Theorems \ref{THM_3-8} and \ref{THM_4-3}.

  Let $\omega=T(1)$.
  To see that $\omega \in \SF(Y)$,
  note that $\tau^{-1}: Y \to X$ is the homeomorphism induced by $T^{-1}$.
  Note that both $\tau$ and $\tau^{-1}$ are coarse homeomorphisms by the assumption of chain-connectedness.
  Let $\lambda = T^{-1}(1)$. Then, we have the representations 
  \[
  T = \omega \cdot (\tau^{-1})^{\ast} \quad \text{and} \quad T^{-1} = \lambda \cdot \tau^{\ast}.
  \]
  The positivity of linear lattice isomorphisms implies that $\omega > 0$ and $\lambda > 0$.
  Indeed, positivity implies that $\omega \geq 0$ and $\lambda \geq 0$.
  If $\omega (y) = 0$ for some $y \in Y$,
  then we have
   \[T(f)(y) =\omega(y) \cdot f(\tau^{-1}(y)) = 0\]
   for every $f \in \SO(X)$,
  which contradicts the surjectivity of $T$, i.e., $\omega > 0$.
  Similarly, we have $\lambda > 0$.
  Thus, both $\omega$ and $\lambda$ satisfy condition $(\sharp 1)$.

  For each $g \in \SO(Y)$,
  let $h=g\circ \tau$.
  Since $\tau$ is a coarse homeomorphism, $h\in \SO(X)$ by Lemma \ref{LEM_4-1}.
  Then we have 
  \[T(h) = (\omega \cdot (\tau^{-1})^{\ast})(h)
    = \omega \cdot (\tau^{-1})^{\ast}(h) = \omega \cdot g.\]
  Since $T(h)\in \SO(Y)$, we have $\omega \cdot g \in \SO(Y)$.
  Similarly, we have $\lambda \cdot f \in \SO(X)$ for every $f \in \SO(X)$.
  Thus, both $\omega$ and $\lambda$ satisfy condition $(\sharp 2)$.

  For every $g\in \SO(Y)$, note that
  \begin{align*}
    (T \circ T^{-1})(g)
    &= T(T^{-1}(g))
    = T(\lambda \cdot (g\circ \tau))\\
    &= \omega \cdot ((\lambda \circ \tau^{-1}) \cdot (g \circ \tau \circ \tau^{-1}))
    = \omega \cdot (\lambda \circ \tau^{-1}) \cdot g.
  \end{align*}
  Since $T \circ T^{-1} = \mathrm{id}_{\SO(Y)}$, taking $g = 1$,
  we have $\omega \cdot (\lambda \circ \tau^{-1}) = 1$.
  Thus, $1/ \omega  = \lambda \circ \tau^{-1} \in \SO(Y)$.
  Then we have $(1/ \omega) \cdot g \in \SO(Y)$ for every $g \in \SO(Y)$.
  Indeed, letting $h = g \circ \tau \in \SO (X)$, we have 
  \begin{align*}
    (\tau^{-1})^{\ast}(\lambda \cdot h) 
    &= (\lambda \circ \tau^{-1}) \cdot (h \circ \tau^{-1}) 
    = (\lambda \circ \tau^{-1}) \cdot g = (1/\omega) \cdot g.
  \end{align*}
  Since $\lambda$ satisfies $(\sharp 2)$,
   we have $\lambda \cdot h \in \SO(X)$.
  Thus, we have $(1/\omega) \cdot g=(\tau^{-1})^{\ast}(\lambda \cdot h) \in \SO(Y)$,
  i.e., $\omega$ satisfies condition $(\sharp 3)$.
\end{proof}

Let $\CH(X,Y)$ denote the set of all coarse homeomorphisms from $X$ to $Y$.
Denote by $\ISO(X,Y)$ (resp. $\ISOs(X, Y)$)
the set of all linear lattice isomorphisms from $\SO(X)$ onto $\SO(Y)$
(resp.  from $\SO^{\ast}(X)$
onto $\SO^{\ast}(Y)$).

\begin{thm}[Representation Theorem]\label{THM_4-11}
  Let $(X, d_X)$ and $(Y, d_Y)$ be chain-connected proper metric spaces.
  For each $\omega \in \SO_{\sharp}(Y)$ and $\tau \in \CH(X,Y)$, define
  $\Phi( \omega, \tau) = \omega \cdot (\tau^{-1})^{\ast}$.
  Then, $\Phi$ induces bijections
  $\SF(Y)\times \CH(X,Y) \to \ISO(X, Y)$
  and
  $\SF^{\ast}(Y) \times \CH(X,Y) \to \ISOs(X, Y)$.
\end{thm}

\begin{proof}
  We only show that $\Phi$ induces bijections
  $\SF(Y) \times \CH(X,Y) \to \ISO(X, Y)$.
  The proof for $\SF^{\ast}(Y) \times \CH(X,Y) \to \ISOs(X, Y)$ follows similarly and is left to the reader.

 First we shall check the well-definedness of $\Phi$.
  Let $\omega \in \SF(Y)$ and $\tau \in \CH(X,Y)$.
For every $f\in \SO(X)$,
 $\Phi  (\omega, \tau)(f)=\omega\cdot (f\circ\tau^{-1})\in \SO(Y)$,
  that is, $\Phi  (\omega, \tau)$ maps $\SO(X)$ into $\SO(Y)$.
Consider $\Psi  (\omega, \tau)=(1/(\omega\circ \tau))\cdot \tau^{\ast}$.
Then for each $g\in \SO(Y)$,
 we have
 \[\Psi(\omega, \tau)(g)=((1/\omega)\circ \tau )\cdot (g\circ \tau)=\tau^{\ast}((1/\omega)\cdot g)
  \in \SO(X).\]
That is, $\Psi(\omega, \tau)$ maps $\SO(Y)$ into $\SO(X)$.
Then we have, for each $f\in \SO(X)$,
\begin{align*}
  \Psi(\omega, \tau)\circ \Phi(\omega, \tau)(f)
  &=\Psi(\omega, \tau)\left(\omega\cdot (f\circ \tau^{-1})\right)\\
  &=(1/(\omega\circ \tau))\cdot\left((\omega\circ \tau)\cdot (f\circ \tau^{-1}\circ \tau)\right)
  =f
\end{align*}
Hence, we have $\Psi(\omega, \tau)\circ \Phi(\omega, \tau)=\textrm{id}_{\SO(X)}$.
Similarly, we have $\Phi(\omega, \tau)\circ \Psi(\omega, \tau)=\textrm{id}_{\SO(Y)}$.
The linearity of $\Phi  (\omega, \tau)$ is obvious.
Thus, we have $\Phi (\omega, \tau) \in \ISO(X, Y)$.
  \par
  Now we shall show that $\Phi $ is a bijection.
  By Theorem \ref{THM_4-10}, $\Phi $ is surjective.
  To see that $\Phi $ is injective,
  let $( \omega, \tau), ( \omega', \tau') \in \SF(Y) \times \CH(X,Y)$.
  Suppose that $\tau\neq \tau'$.
  Let $y\in Y$ be a point such that $\tau^{-1}(y) \neq (\tau')^{-1}(y)$.
  Then we can take $f \in \SO^{\ast}(X)$ such that
  $f( \tau^{-1}(y))=0$ and $f( (\tau')^{-1}(y)) = 1$.
  Hence, we have 
  \[\left(\Phi ( \omega, \tau )(f)\right)(y) = \omega(y) \cdot f( \tau^{-1}(y)) = 0,\]
  whereas 
  \[\left(\Phi ( \omega', \tau' )(f) \right)(y)
   = \omega' (y) \cdot f( (\tau')^{-1}(y)) = \omega'(y) > 0,\]
  i.e., $\Phi ( \omega, \tau) \neq \Phi ( \omega', \tau')$.
  Suppose that $\omega(y) \neq \omega'(y)$ for some $y \in Y$.
  Then we have
  \[\left(\Phi ( \omega, \tau )(f)\right)(y) = \omega(y) \cdot f( \tau^{-1} (y))
    \neq \omega'(y) \cdot f(\tau^{-1} (y)) = \left(\Phi ( \omega', \tau )(f) \right)(y)\]
  whenever $f(\tau^{-1} (y)) \neq 0$,
  i.e., $\Phi ( \omega, \tau ) \neq \Phi (\omega', \tau)$.
  This shows that $\Phi $ is injective, and thus, $\Phi $ is a bijection.
\end{proof}

\begin{corollary}\label{COR_4-12}
  If $(X, d_X)$ is a chain-connected proper metric space
  then $\ISO(Y,X)$ is a subset of $\ISOs(Y, X)$.
  In particular, $\SF(X)$ is a sublattice of $\SF^{\ast}(X)$.
\end{corollary}

\begin{proof}
  By Lemma \ref{LEM_4-4},
   it follows that $\ISO(Y,X) \subset \ISOs(Y, X)$.
  Thus, we have $\SF(X)\subset \SF^{\ast}(X)$ by Theorems \ref{THM_4-10} and \ref{THM_4-11}.
  Hence, $\SF(X)$ is a sublattice of $\SF^{\ast}(X)$ by Lemma \ref{LEM_4-7}.
\end{proof}

\begin{remark}\label{REM_4-13}
  The lattice $\SF(X)$ is not equal to $\SF^{\ast}(X)$ in general.
  Let $X = [0, \infty)$ be the half-open interval with the metric $d(x,y) = |x-y|$.
  Let $\omega(x) = 2 + \sin \sqrt{x}$ and $f(x) = \sqrt{x}$ for all $x \in X$.
  It is easy to see that $\omega \in \SO^{\ast}(X)$ and $f \in \SO(X)$.
  Because $\omega$ is a strongly positive element of $\SO^{\ast}(X)$,
  $\omega$ is an element of $\SF^{\ast}(X)$ by Lemma \ref{LEM_4-8}.
  Note that
 \[\frac{d}{dx}\left(\omega\cdot f\right)=\frac{\cos \sqrt{x}}{2}+ \frac{2 + \sin \sqrt{x}}{2 \sqrt{x}}. \]
 Observe that, for any $R>0$, there exists $x\in X$ such that
  $|x|>R$ and $\frac{d}{dx}\left(\omega\cdot f\right)(x)>1/4$.
  Since $(\cos \sqrt{x} )/2$ is slowly oscillating,
  it follows that $\omega \cdot f$ is not slowly oscillating.
  Hence, $\omega $ is not contained in $\SF(X)$, i.e.,
  $\SF(X) \subsetneq \SF^{\ast}(X)$.
  \hfill$\square$
\end{remark}


\end{document}